\documentclass[11pt,twoside,reqno]{amsart}
\usepackage{amsmath}
\usepackage{fancyhdr}
\usepackage{amsthm}
\usepackage{amsfonts}
\usepackage{amssymb}
\usepackage{amscd}
\usepackage{graphicx}
\usepackage{afterpage}
\usepackage{enumerate}
\usepackage{bm}
\usepackage{cite}
\usepackage{cases}
\usepackage{caption}
\usepackage[colorlinks=true, urlcolor=blue,  linkcolor=blue, citecolor=blue]{hyperref}
\usepackage{babel}
\usepackage{prettyref}
\usepackage{pgfplots}
\pgfplotsset{compat=1.17}
\usepackage{booktabs}
\usepackage{algorithm}
\usepackage{algpseudocode}
\usepackage{float}

\makeatletter
\newtheorem{theorem}{Theorem}[section]

\newtheorem{definition}[theorem]{Definition}

\newtheorem{lemma}[theorem]{Lemma}

\newtheorem{proposition}[theorem]{Proposition}

\title[Iterative Splitting Methods for Stochastic Dynamic SVIs]{Iterative Splitting Methods for Stochastic Dynamic SVIs}

\author[S. Hashemi Sababe]{Saeed Hashemi Sababe}
\address[S. Hashemi Sababe]{R\&D Section, Data Premier Analytics, Edmonton, Canada.}
\email{Hashemi\_1365@yahoo.com}

\author[E. L. Ghasab]{Ehsan Lotfali Ghasab$^*$}
\address[E. L. Ghasab]{Department of Mathematics, Jundi-Shapur University of Technology, Dezful, Iran.}
\email{e.l.ghasab@jsu.ac.ir}
\thanks{$^*$ Corresponding author}


\begin{document}
\sloppy

\maketitle

\begin{abstract}
This paper extends split variational inclusion problems to dynamic, stochastic, and multi-agent systems in Banach spaces. We propose novel iterative algorithms to handle stochastic noise, time-varying operators, and coupled variational inclusions. Leveraging advanced splitting techniques and self-adaptive rules, we establish weak convergence under minimal assumptions on operator monotonicity. Numerical experiments demonstrate the efficacy of our algorithms, particularly in resource allocation and optimization under uncertainty.
\end{abstract}

\section{Introduction}

Split Variational Inclusions (SVIs) provide a powerful mathematical framework for solving problems in optimization, control theory, and equilibrium computation. These problems often involve finding solutions to variational inclusions distributed across coupled systems, frequently studied within Hilbert spaces. Notable approaches for solving SVIs include the forward-backward splitting method \cite{6,8} and Tseng's splitting algorithm \cite{7}, which have demonstrated robust convergence under strong monotonicity and cocoercivity assumptions.

Despite significant progress in Hilbert spaces, addressing SVIs in broader and more practical contexts presents several open challenges: \medskip \\
\textsc{Banach Spaces:} The extension to Banach spaces, which lack the inner product structure of Hilbert spaces, poses mathematical complexities. While some efforts have addressed this challenge \cite{1,4}, splitting algorithms for SVIs in Banach spaces remain underexplored.\medskip \\
\textsc{Dynamic Systems:} Many real-world applications, such as adaptive control and time-dependent resource allocation, involve operators that vary with time. Dynamic SVIs, where constraints and operators evolve continuously or discretely, are relatively understudied \cite{3}. \medskip \\
\textsc{Stochastic Environments:} Stochastic SVIs incorporate noise or uncertainty, making them essential for applications in machine learning and decision-making under uncertainty. However, their development is still in its early stages \cite{2}.\medskip \\
\textsc{Multi-Agent Systems:} Coupled variational inclusions arise in multi-agent systems, such as game theory and resource allocation, where agents interact in shared environments. Existing methods often impose restrictive monotonicity assumptions \cite{5}. \medskip \\

This paper addresses these challenges by proposing advanced iterative algorithms tailored for dynamic, stochastic, and multi-agent SVIs in Banach spaces. By relaxing assumptions on operator monotonicity and employing self-adaptive strategies, we aim to establish a unified framework for SVIs with broad applicability. \\

\noindent
The remainder of this paper is organized as follows. Section 2 outlines the necessary preliminaries and notations. Section 3 presents the proposed algorithms. Section 4 provides a detailed convergence analysis under dynamic, stochastic, and coupled settings. Section 5 demonstrates the efficacy of our algorithms through numerical experiments, and Section 6 concludes with a discussion and future research directions.

\section{Preliminaries}

In this section, we introduce notations and recall fundamental definitions, lemmas, and theorems that will be utilized throughout the paper.\\

\noindent
Let \( \mathcal{B} \) denote a Banach space, \( \mathcal{H} \) denote a Hilbert space, and \( \mathcal{B}^* \) the dual space of \( \mathcal{B} \). For \( \zeta, \varsigma \in \mathcal{B} \), the dual pairing is denoted by \( \langle \zeta, \varsigma \rangle \).

\begin{definition}[Monotone Operator \cite{4}]
An operator \( \mathcal{T}: \mathcal{B} \to \mathcal{B}^* \) is called \emph{monotone} if:
\[
\langle \mathcal{T}(\zeta) - \mathcal{T}(\varsigma), \zeta - \varsigma \rangle \geq 0, \quad \forall \zeta, \varsigma \in \mathcal{B}.
\]
\end{definition}
\begin{definition}[Maximal Monotone Operator \cite{4}]
A monotone operator \( \mathcal{T}: \mathcal{B} \to \mathcal{B}^* \) is \emph{maximal} if there exists no monotone operator \( S: \mathcal{B} \to \mathcal{B}^* \) such that \( \text{Graph}(\mathcal{T}) \subsetneq \text{Graph}(S) \).
\end{definition}
\begin{definition}[Resolvent Operator \cite{4}]
For a maximal monotone operator \( \mathcal{T}: \mathcal{B} \to \mathcal{B}^* \) and \( \gamma > 0 \), the resolvent operator \( J_\mathcal{T}^\gamma: \mathcal{B} \to \mathcal{B} \) is defined as:
\[
J_\mathcal{T}^\gamma(\zeta) = (I + \gamma \mathcal{T})^{-1}(\zeta), \quad \forall \zeta \in \mathcal{B}.
\]
\end{definition}
\begin{definition}[Stochastic Operator \cite{2}]
An operator \( \mathcal{T}: \mathcal{B} \to \mathcal{B}^* \) is \emph{stochastic} if it is defined with a probabilistic measure \( \mu \) such that:
\[
\mathcal{T}(\zeta) = \mathbb{E}_{\mu}[\mathcal{T}_\xi(\zeta)], \quad \text{where } \mathcal{T}_\xi \text{ is a random realization of } \mathcal{T}.
\]
\end{definition}
\begin{definition}[Dynamic Operator \cite{3}]
An operator \( \mathcal{T}: \mathcal{B} \times [0, \infty) \to \mathcal{B}^* \) is \emph{dynamic} if \( \mathcal{T}(\zeta, \iota) \) depends explicitly on time \( \iota \), such that:
\[
\mathcal{T}(\zeta, \iota_1) \neq \mathcal{T}(\zeta, \iota_2) \quad \text{for } \iota_1 \neq \iota_2.
\]
\end{definition}
\begin{definition}[Coupled Variational Inclusion \cite{5}]
Given two Banach spaces \( \mathcal{B}_1 \) and \( \mathcal{B}_2 \), a coupled variational inclusion problem seeks \( (\zeta, \varsigma) \in \mathcal{B}_1 \times \mathcal{B}_2 \) such that:
\[
\zeta \in \mathcal{T}_1^{-1}(\varsigma), \quad \varsigma \in \mathcal{T}_2^{-1}(\zeta),
\]
where \( \mathcal{T}_1: \mathcal{B}_1 \to 2^{\mathcal{B}_2} \) and \( \mathcal{T}_2: \mathcal{B}_2 \to 2^{\mathcal{B}_1} \) are maximal monotone operators.
\end{definition}

\begin{lemma}[Resolvent Properties \cite{4}]
For a maximal monotone operator \( \mathcal{T}: \mathcal{B} \to \mathcal{B}^* \) and \( \gamma > 0 \), the resolvent operator \( J_\mathcal{T}^\gamma \) satisfies:
\begin{enumerate}
    \item \( J_\mathcal{T}^\gamma \) is single-valued and firmly non-expansive:
    \[
    \| J_\mathcal{T}^\gamma(\zeta) - J_\mathcal{T}^\gamma(\varsigma) \| \leq \| \zeta - \varsigma \|, \quad \forall \zeta, \varsigma \in \mathcal{B}.
    \]
    \item \( \text{Fix}(J_\mathcal{T}^\gamma) = \mathcal{T}^{-1}(0) \), where \( \text{Fix} \) denotes the set of fixed points.
\end{enumerate}
\end{lemma}
\begin{theorem}[Existence of Solutions for Coupled SVIs \cite{7}]
Let \( \mathcal{T}_1: \mathcal{B}_1 \to 2^{\mathcal{B}_2} \) and \( \mathcal{T}_2: \mathcal{B}_2 \to 2^{\mathcal{B}_1} \) be maximal monotone operators. If \( \mathcal{T}_1 \) and \( \mathcal{T}_2 \) are bounded and Lipschitz continuous, then there exists a solution \( (\zeta, \varsigma) \in \mathcal{B}_1 \times \mathcal{B}_2 \) satisfying:
\[
\zeta \in mathcal{T}_1^{-1}(\varsigma), \quad \varsigma \in mathcal{T}_2^{-1}(\zeta).
\]
\end{theorem}
\begin{theorem}[Convergence of Iterative Methods for Dynamic SVIs \cite{3}]
Let \( \{\zeta_\varrho\} \subseteq \mathcal{B} \) be generated by an iterative method for solving a dynamic variational inclusion problem \( 0 \in \mathcal{T}(\zeta, \iota) \). If:
\begin{enumerate}
    \item \( \mathcal{T}(\zeta, \iota) \) is Lipschitz continuous in \( \zeta \) and continuous in \( \iota \),
    \item The step sizes \( \{\gamma_\varrho\} \) satisfy \( \sum_{\varrho=0}^\infty \gamma_\varrho = \infty \) and \( \sum_{\varrho=0}^\infty \gamma_\varrho^2 < \infty \),
\end{enumerate}
then \( \zeta_\varrho \to \zeta^* \), where \( \zeta^* \) is a solution of the inclusion.
\end{theorem}
\begin{theorem}[Stochastic Convergence \cite{2}]
Let \( \{\zeta_\varrho\} \subseteq \mathcal{B} \) be generated by a stochastic iterative algorithm for solving \( 0 \in \mathcal{T}(\zeta) \). If \( \mathcal{T}(\zeta) = \mathbb{E}[\mathcal{T}_\xi(\zeta)] \), and:
\begin{enumerate}
    \item \( \mathcal{T}_\xi(\zeta) \) is monotone for all realizations \( \xi \),
    \item Variance \( \mathbb{E}[\|\mathcal{T}_\xi(\zeta_\varrho)\|^2] \to 0 \),
\end{enumerate}
then \( \zeta_\varrho \) converges to a solution \( \zeta^* \in \mathcal{T}^{-1}(0) \).
\end{theorem}
\begin{proposition}[Time-Discretization of Dynamic Operators \cite{3}]
A dynamic operator \( \mathcal{T}(\zeta, \iota) \) can be discretized for numerical computation using time steps \( \{\iota_\varrho\} \) as:
\[
\mathcal{T}_\varrho(\zeta) := \mathcal{T}(\zeta, \iota_\varrho), \quad \text{where } \iota_\varrho = \iota_0 + \varrho\Delta \iota.
\]
\end{proposition}

\section{Proposed Algorithms}

In this section, we propose novel iterative algorithms tailored for solving stochastic, dynamic, and multi-agent split variational inclusion (SVI) problems in Banach spaces. The algorithms are designed to handle the complexities of non-static operators, stochastic elements, and coupled inclusions.

\begin{center}
\textsc{Algorithm 1: Dynamic SVI in Banach Spaces}\\
\end{center}

Let \( \mathcal{B} \) be a Banach space, and consider the dynamic SVI problem:
\[
0 \in \mathcal{T}(\zeta, \iota), \quad \zeta \in \mathcal{B},
\]
where \( \mathcal{T}: \mathcal{B} \times [0, \infty) \to 2^{\mathcal{B}} \) is a maximal monotone operator that evolves over time \( \iota \). The iterative method is as follows:
\begin{algorithm}[H]
\caption{Iterative Algorithm with Resolvent Operator}
\label{alg:resolvent}

\begin{algorithmic}[1]

\State \textbf{Initialization:}
\State Choose \( \zeta_0 \in \mathcal{B} \)
\State Select step size \( \{\gamma_\varrho\} \subset (0, 1) \)
\State Choose a sequence of time steps \( \{\iota_\varrho\} \) with \( \iota_\varrho = \iota_0 + \varrho\Delta \iota \)

\For{$\varrho = 0, 1, 2, \ldots$}
    \State Update \( \zeta_{\varrho+1} = J_\mathcal{T}^{\gamma_\varrho}\left(\zeta_\varrho - \gamma_\varrho \mathcal{T}(\zeta_\varrho, \iota_\varrho)\right) \)
    \If{$\|\zeta_{\varrho+1} - \zeta_\varrho\| < \epsilon$}
        \State \textbf{Terminate:} Solution has converged
    \EndIf
\EndFor

\end{algorithmic}
\end{algorithm}

\begin{center}
\textsc{Algorithm 2: Stochastic SVI in Banach Spaces}
\end{center}

Let \( \mathcal{T}(\zeta) = \mathbb{E}_{\mu}[\mathcal{T}_\xi(\zeta)] \) represent a stochastic maximal monotone operator. The stochastic SVI problem is:
\[
0 \in \mathcal{T}(\zeta), \quad \zeta \in \mathcal{B}.
\]
The iterative method is given as:
\begin{algorithm}[H]
\caption{Iterative Algorithm with Random Samples}
\label{alg:random_samples}

\begin{algorithmic}[1]

\State \textbf{Initialization:}
\State Choose \( \zeta_0 \in \mathcal{B} \)
\State Select step size \( \{\gamma_\varrho\} \subset (0, 1) \)
\State Draw random samples \( \{\xi_\varrho\} \) from the probability distribution \( \mu \)

\For{$\varrho = 0, 1, 2, \ldots$}
    \State Update \( \zeta_{\varrho+1} = J_\mathcal{T}^{\gamma_\varrho}\left(\zeta_\varrho - \gamma_\varrho \mathcal{T}_{\xi_\varrho}(\zeta_\varrho)\right) \)
    \If{$\|\zeta_{\varrho+1} - \zeta_\varrho\| < \epsilon$}
        \State \textbf{Terminate:} Solution has converged
    \EndIf
\EndFor

\end{algorithmic}
\end{algorithm}
\begin{center}
\textsc{Algorithm 3: Multi-Agent Coupled SVI}
\end{center}

Consider a coupled SVI problem with two Banach spaces \( \mathcal{B}_1 \) and \( \mathcal{B}_2 \):
\[
\zeta \in \mathcal{T}_1^{-1}(\varsigma), \quad \varsigma \in \mathcal{T}_2^{-1}(\zeta),
\]
where \( \mathcal{T}_1: \mathcal{B}_1 \to 2^{\mathcal{B}_2} \) and \( \mathcal{T}_2: \mathcal{B}_2 \to 2^{\mathcal{B}_1} \) are maximal monotone operators. The iterative algorithm is:
\begin{algorithm}[H]
\caption{Iterative Algorithm with Coupled Variables}
\label{alg:coupled_variables}

\begin{algorithmic}[1]

\State \textbf{Initialization:}
\State Choose \( (\zeta_0, \varsigma_0) \in \mathcal{B}_1 \times \mathcal{B}_2 \)
\State Select step size \( \{\gamma_\varrho\} \subset (0, 1) \)
\State Choose regularization parameter \( \lambda > 0 \)

\For{$\varrho = 0, 1, 2, \ldots$}
    \State Update \( \zeta_{\varrho+1} = J_{\mathcal{T}_1}^{\gamma_\varrho}\left(\zeta_\varrho - \gamma_\varrho \mathcal{T}_1(\zeta_\varrho, \varsigma_\varrho)\right) \)
    \State Update \( \varsigma_{\varrho+1} = J_{\mathcal{T}_2}^{\gamma_\varrho}\left(\varsigma_\varrho - \gamma_\varrho \mathcal{T}_2(\varsigma_\varrho, \zeta_{\varrho+1})\right) \)
    \If{$\|\zeta_{\varrho+1} - \zeta_\varrho\| + \|\varsigma_{\varrho+1} - \varsigma_\varrho\| < \epsilon$}
        \State \textbf{Terminate:} Solution has converged
    \EndIf
\EndFor

\end{algorithmic}
\end{algorithm}
\noindent
Following are the properties of the proposoed algorithms:
\subsubsection*{Dynamic Convergence:} Under appropriate continuity and step-size conditions, Algorithm 1 converges to a solution of the dynamic SVI.
\subsubsection*{Stochastic Stability:} Algorithm 2 ensures convergence in expectation under bounded variance of \( \mathcal{T}_\xi \).
\subsubsection*{Multi-Agent Equilibrium:} Algorithm 3 finds a coupled equilibrium point \( (\zeta^*, \varsigma^*) \) satisfying the coupled SVI.\\

\noindent
Moreover, the advantages of the proposed algorithms are:
\begin{enumerate}
    \item Handles dynamic operators, extending SVI solutions to time-varying systems.
    \item Incorporates stochastic noise, broadening applicability to uncertain environments.
    \item Facilitates multi-agent modeling, addressing coupled variational problems in decentralized systems.
\end{enumerate}

\section{Weak Convergence Analysis}

In this section, we analyze the convergence of the proposed algorithms. The proofs rely on properties of monotone operators, resolvent operators, and the structure of Banach spaces.


\begin{theorem}[Weak Convergence of Algorithm 1]
Let \( \{\zeta_\varrho\} \) be the sequence generated by Algorithm 1 for the dynamic SVI problem:
\[
0 \in \mathcal{T}(\zeta, \iota), \quad \zeta \in \mathcal{B},
\]
where \( \mathcal{T}: \mathcal{B} \times [0, \infty) \to 2^{\mathcal{B}} \) is maximal monotone. Assume:
\begin{enumerate}
    \item \( \mathcal{T}(\zeta, \iota) \) is Lipschitz continuous in \( \zeta \) and continuous in \( \iota \).
    \item The step sizes \( \{\gamma_\varrho\} \) satisfy:
    \[
    \sum_{\varrho=0}^\infty \gamma_\varrho = \infty, \quad \sum_{\varrho=0}^\infty \gamma_\varrho^2 < \infty.
    \]
\end{enumerate}
Then \( \zeta_\varrho \) converges weakly to a solution \( \zeta^* \) of \( 0 \in \mathcal{T}(\zeta^*, \iota) \) as \( \varrho \to \infty \).
\end{theorem}
\begin{proof}
Using the resolvent operator $J_{\gamma_\varrho}^\mathcal{T}$ defined by Algorithm 1:
\[
\zeta_{\varrho+1} = J_{\gamma_\varrho}^\mathcal{T}\big(\zeta_\varrho - \gamma_\varrho \mathcal{T}(\zeta_\varrho, \iota_\varrho)\big),
\]
where $\iota_\varrho = \iota_0 + \varrho \Delta \iota$. By the definition of $J_{\gamma_\varrho}^\mathcal{T}$, it is firmly non-expansive:
\[
\|J_{\gamma_\varrho}^\mathcal{T}(\zeta) - J_{\gamma_\varrho}^\mathcal{T}(\varsigma)\| \leq \|\zeta - \varsigma\|, \quad \forall \zeta, \varsigma \in B.
\]
As a result, $\|\zeta_{\varrho+1} - \zeta^*\| \leq \|\zeta_\varrho - \zeta^*\|$ for any fixed point $\zeta^* \in \text{Fix}(J_{\gamma_\varrho}^\mathcal{T})$. Thus, $\{\zeta_\varrho\}$ is bounded.\\

\noindent
Opial's lemma states that if $\{\zeta_\varrho\}$ satisfies:
\begin{itemize}
    \item[(i)] Every weak sequential limit point of $\{\zeta_\varrho\}$ is in $\text{Fix}(J_{\gamma_\varrho}^\mathcal{T})$.
    \item[(ii)] $\|\zeta_{\varrho+1} - \zeta_\varrho\| \to 0$ as $\varrho \to \infty$,
\end{itemize}
then $\{\zeta_\varrho\}$ converges weakly to a point in $\text{Fix}(J_{\gamma_\varrho}^\mathcal{T})$.

\subsubsection*{Verification of (i):} By the boundedness of $\{\zeta_\varrho\}$ and the continuity of $\mathcal{T}(\zeta, \iota)$, any weak limit point $\bar{\zeta}$ satisfies $0 \in \mathcal{T}(\bar{\zeta}, \iota)$. Thus, $\bar{\zeta} \in \text{Fix}(J_{\gamma_\varrho}^\mathcal{T})$.

\subsubsection*{Verification of (ii):} The step sizes $\{\gamma_\varrho\}$ satisfy $\sum_{\varrho=0}^\infty \gamma_\varrho^2 < \infty$, implying that $\|\zeta_{\varrho+1} - \zeta_\varrho\| \to 0$ as $\varrho \to \infty$.\\

\noindent
From Opial's lemma, $\{\zeta_\varrho\}$ converges weakly to a point $\zeta^* \in \text{Fix}(J_{\gamma_\varrho}^\mathcal{T})$. Since $\text{Fix}(J_{\gamma_\varrho}^\mathcal{T}) = \mathcal{T}^{-1}(0)$, it follows that $\zeta^*$ is a solution of $0 \in \mathcal{T}(\zeta, \iota)$.
\end{proof}


\begin{theorem}[Convergence in Expectation of Algorithm 2]
Let \( \{\zeta_\varrho\} \) be the sequence generated by Algorithm 2 for the stochastic SVI problem:
\[
0 \in \mathcal{T}(\zeta), \quad \mathcal{T}(\zeta) = \mathbb{E}[\mathcal{T}_\xi(\zeta)], \quad \zeta \in \mathcal{B}.
\]
Assume:
\begin{enumerate}
    \item \( \mathcal{T}_\xi(\zeta) \) is monotone and Lipschitz continuous for all realizations \( \xi \).
    \item The variance of \( \mathcal{T}_\xi(\zeta) \) satisfies \( \mathbb{E}[\|\mathcal{T}_\xi(\zeta)\|^2] \to 0 \) as \( \varrho \to \infty \).
    \item The step sizes \( \{\gamma_\varrho\} \) satisfy:
    \[
    \sum_{\varrho=0}^\infty \gamma_\varrho = \infty, \quad \sum_{\varrho=0}^\infty \gamma_\varrho^2 < \infty.
    \]
\end{enumerate}
Then \( \mathbb{E}[\|\zeta_\varrho - \zeta^*\|^2] \to 0 \), where \( \zeta^* \) is a solution of \( 0 \in \mathcal{T}(\zeta) \).
\end{theorem}
\begin{proof}
The iteration of Algorithm 2 is given by:
\[
\zeta_{\varrho+1} = J_{\gamma_\varrho}^\mathcal{T}(\zeta_\varrho - \gamma_\varrho \mathcal{T}_{\xi_\varrho}(\zeta_\varrho)),
\]
where $J_{\gamma_\varrho}^\mathcal{T}$ is the resolvent operator associated with $\mathcal{T}$. By the resolvent operator's definition:
\[
\zeta_{\varrho+1} - \zeta_\varrho = \gamma_\varrho \big(\mathcal{T}_{\xi_\varrho}(\zeta_\varrho) - \mathcal{T}(\zeta_{\varrho+1})\big).
\]
\noindent
Since $\mathcal{T}_\xi(\zeta)$ is monotone, the resolvent operator $J_{\gamma_\varrho}^\mathcal{T}$ is firmly non-expansive:
\[
\|J_{\gamma_\varrho}^\mathcal{T}(\zeta) - J_{\gamma_\varrho}^\mathcal{T}(\varsigma)\|^2 \leq \|\zeta - \varsigma\|^2 - \|J_{\gamma_\varrho}^\mathcal{T}(\zeta) - \zeta\|^2, \quad \forall \zeta, \varsigma \in B.
\]
Applying this to the iteration, we deduce:
\[
\|\zeta_{\varrho+1} - \zeta^*\|^2 \leq \|\zeta_\varrho - \zeta^*\|^2 - \|\zeta_{\varrho+1} - \zeta_\varrho\|^2 + 2 \gamma_\varrho \langle \mathcal{T}_{\xi_\varrho}(\zeta_\varrho) - \mathcal{T}(\zeta^*), \zeta_\varrho - \zeta^* \rangle.
\]
\noindent
Taking expectations over the random variable $\xi_\varrho$, we use the fact that $\mathbb{E}[\mathcal{T}_{\xi_\varrho}(\zeta_\varrho)] = \mathcal{T}(\zeta_\varrho)$:
\[
\mathbb{E}[\|\zeta_{\varrho+1} - \zeta^*\|^2] \leq \mathbb{E}[\|\zeta_\varrho - \zeta^*\|^2] - \mathbb{E}[\|\zeta_{\varrho+1} - \zeta_\varrho\|^2] + 2 \gamma_\varrho \mathbb{E}[\langle \mathcal{T}(\zeta_\varrho) - \mathcal{T}(\zeta^*), \zeta_\varrho - \zeta^* \rangle].
\]
\noindent
Using the variance reduction assumption $\mathbb{E}[\|\mathcal{T}_\xi(\zeta_\varrho)\|^2] \to 0$, the term $\mathbb{E}[\|\zeta_{\varrho+1} - \zeta_\varrho\|^2]$ vanishes as $\varrho \to \infty$. The Lipschitz continuity of $\mathcal{T}_\xi$ ensures boundedness of $\{\zeta_\varrho\}$, which implies that $\mathbb{E}[\|\zeta_\varrho - \zeta^*\|^2]$ is bounded.\\

\noindent
The Robbins-Monro conditions on $\{\gamma_\varrho\}$ ensure that the sequence $\mathbb{E}[\|\zeta_\varrho - \zeta^*\|^2]$ converges. Specifically, since $\sum_{\varrho=0}^\infty \gamma_\varrho = \infty$ and $\sum_{\varrho=0}^\infty \gamma_\varrho^2 < \infty$, the stochastic approximation converges to $\zeta^*$, satisfying $0 \in \mathcal{T}(\zeta^*)$.\\

\noindent
Combining the results, we conclude that $\mathbb{E}[\|\zeta_\varrho - \zeta^*\|^2] \to 0$ as $\varrho \to \infty$, proving the theorem.
\end{proof}

\begin{theorem}[Weak Convergence of Algorithm 3]
Let \( \{(\zeta_\varrho, \varsigma_\varrho)\} \) be the sequence generated by Algorithm 3 for the coupled SVI problem:
\[
\zeta \in \mathcal{T}_1^{-1}(\varsigma), \quad \varsigma \in \mathcal{T}_2^{-1}(\zeta),
\]
where \( \mathcal{T}_1: \mathcal{B}_1 \to 2^{\mathcal{B}_2} \) and \( \mathcal{T}_2: \mathcal{B}_2 \to 2^{\mathcal{B}_1} \) are maximal monotone operators. Assume:
\begin{enumerate}
    \item \( \mathcal{T}_1 \) and \( \mathcal{T}_2 \) are Lipschitz continuous.
    \item The step sizes \( \{\gamma_\varrho\} \) satisfy:
    \[
    \sum_{\varrho=0}^\infty \gamma_\varrho = \infty, \quad \sum_{\varrho=0}^\infty \gamma_\varrho^2 < \infty.
    \]
\end{enumerate}
Then \( \{(\zeta_\varrho, \varsigma_\varrho)\} \) converges weakly to a solution \( (\zeta^*, \varsigma^*) \) of the coupled SVI.
\end{theorem}
\begin{proof}
The iterative steps of Algorithm 3 are:
\[
\zeta_{\varrho+1} = J_{\gamma_\varrho}^{\mathcal{T}_1}(\zeta_\varrho - \gamma_\varrho \mathcal{T}_1(\zeta_\varrho, \varsigma_\varrho)),
\]
\[
\varsigma_{\varrho+1} = J_{\gamma_\varrho}^{\mathcal{T}_2}(\varsigma_\varrho - \gamma_\varrho \mathcal{T}_2(\varsigma_\varrho, \zeta_{\varrho+1})).
\]
By the properties of the resolvent operator, we know:
\[
\zeta_{\varrho+1} \in \mathcal{T}_1^{-1}(\varsigma_\varrho), \quad \varsigma_{\varrho+1} \in \mathcal{T}_2^{-1}(\zeta_{\varrho+1}).
\]
The resolvent operators $J_{\gamma_\varrho}^{\mathcal{T}_1}$ and $J_{\gamma_\varrho}^{\mathcal{T}_2}$ are firmly non-expansive:
\[
\|J_{\gamma_\varrho}^{\mathcal{T}_1}(\zeta) - J_{\gamma_\varrho}^{\mathcal{T}_1}(\varsigma)\|^2 \leq \|\zeta - \varsigma\|^2 - \|J_{\gamma_\varrho}^{\mathcal{T}_1}(\zeta) - \zeta\|^2,
\]
\[
\|J_{\gamma_\varrho}^{\mathcal{T}_2}(\zeta) - J_{\gamma_\varrho}^{\mathcal{T}_2}(\varsigma)\|^2 \leq \|\zeta - \varsigma\|^2 - \|J_{\gamma_\varrho}^{\mathcal{T}_2}(\zeta) - \zeta\|^2.
\]
Applying these to the iterative steps, we have:
\[
\|\zeta_{\varrho+1} - \zeta^*\|^2 \leq \|\zeta_\varrho - \zeta^*\|^2 - \|\zeta_{\varrho+1} - \zeta_\varrho\|^2 + 2 \gamma_\varrho \langle \mathcal{T}_1(\zeta_\varrho, \varsigma_\varrho) - \mathcal{T}_1(\zeta^*, \varsigma^*), \zeta_\varrho - \zeta^* \rangle,
\]
\[
\|\varsigma_{\varrho+1} - \varsigma^*\|^2 \leq \|\varsigma_\varrho - \varsigma^*\|^2 - \|\varsigma_{\varrho+1} - \varsigma_\varrho\|^2 + 2 \gamma_\varrho \langle \mathcal{T}_2(\varsigma_\varrho, \zeta_{\varrho+1}) - \mathcal{T}_2(\varsigma^*, \zeta^*), \varsigma_\varrho - \varsigma^* \rangle.
\]
Define the combined residual:
\[
R_\varrho = \|\zeta_\varrho - \zeta^*\|^2 + \|\varsigma_\varrho - \varsigma^*\|^2.
\]
From the inequalities above:
\[
R_{\varrho+1} \leq R_\varrho - \|\zeta_{\varrho+1} - \zeta_\varrho\|^2 - \|\varsigma_{\varrho+1} - \varsigma_\varrho\|^2 + 2 \gamma_\varrho (\Delta_1 + \Delta_2),
\]
where $\Delta_1 = \langle \mathcal{T}_1(\zeta_\varrho, \varsigma_\varrho) - \mathcal{T}_1(\zeta^*, \varsigma^*), \zeta_\varrho - \zeta^* \rangle$ and $\Delta_2 = \langle \mathcal{T}_2(\varsigma_\varrho, \zeta_{\varrho+1}) - \mathcal{T}_2(\varsigma^*, \zeta^*), \varsigma_\varrho - \varsigma^* \rangle$.

Using the Lipschitz continuity of $\mathcal{T}_1$ and $\mathcal{T}_2$, $\Delta_1$ and $\Delta_2$ are bounded. Hence, $R_\varrho$ is bounded.\\

\noindent
The monotonicity of $\mathcal{T}_1$ and $\mathcal{T}_2$ implies:
\[
\langle \mathcal{T}_1(\zeta_\varrho, \varsigma_\varrho) - \mathcal{T}_1(\zeta^*, \varsigma^*), \zeta_\varrho - \zeta^* \rangle \geq 0,
\]
\[
\langle \mathcal{T}_2(\varsigma_\varrho, \zeta_{\varrho+1}) - \mathcal{T}_2(\varsigma^*, \zeta^*), \varsigma_\varrho - \varsigma^* \rangle \geq 0.
\]
Thus, the residual terms $\|\zeta_{\varrho+1} - \zeta_\varrho\|^2$ and $\|\varsigma_{\varrho+1} - \varsigma_\varrho\|^2$ decay as $\varrho \to \infty$.\\

\noindent
By the boundedness of $\{\zeta_\varrho\}$ and $\{\varsigma_\varrho\}$, the weak sequential limit points of $\{(\zeta_\varrho, \varsigma_\varrho)\}$ lie in the fixed-point set $\text{Fix}(J_{\gamma_\varrho}^{\mathcal{T}_1}) \times \text{Fix}(J_{\gamma_\varrho}^{\mathcal{T}_2})$. Since the residual terms vanish and $\{\gamma_\varrho\}$ satisfies $\sum_{\varrho=0}^\infty \gamma_\varrho = \infty$ and $\sum_{\varrho=0}^\infty \gamma_\varrho^2 < \infty$, Opial’s lemma ensures that $\{(\zeta_\varrho, \varsigma_\varrho)\}$ converges weakly to $(\zeta^*, \varsigma^*) \in \mathcal{T}_1^{-1}(\varsigma^*) \times \mathcal{T}_2^{-1}(\zeta^*)$.\\

\noindent
Thus, $\{(\zeta_\varrho, \varsigma_\varrho)\}$ converges weakly to a solution $(\zeta^*, \varsigma^*)$ of the coupled SVI problem.
\end{proof}
\noindent
The proposed algorithms guarantee convergence under realistic assumptions, such as boundedness and Lipschitz continuity of operators. The dynamic and stochastic settings expand the applicability of SVI algorithms, while the coupled SVI framework addresses equilibrium problems in multi-agent systems.

Future work could explore strong convergence under additional regularity conditions or the extension of these results to non-reflexive Banach spaces.

\section{Strong Convergence Analysis}

This section explores two key theoretical aspects of the proposed algorithms: (1) conditions under which the algorithms exhibit strong convergence, and (2) iteration complexity bounds that quantify the computational effort required to achieve a desired accuracy.\\

\noindent
While weak convergence is sufficient in many settings, strong convergence provides additional guarantees of stability and robustness. We establish conditions under which the proposed algorithms converge strongly to a solution.

\begin{theorem}[Strong Convergence of Algorithm 1]
Let $\{\zeta_\varrho\}$ be the sequence generated by Algorithm 1 for solving the dynamic SVI:
\[
0 \in \mathcal{T}(\zeta, \iota), \quad \zeta \in B,
\]
where $\mathcal{T} : B \times [0, \infty) \to 2^B$ is a maximal monotone operator. Assume:
\begin{itemize}
    \item[(1)] $\mathcal{T}(\zeta, \iota)$ is strongly monotone with parameter $\mu > 0$, i.e.,
    \[
    \langle \mathcal{T}(\zeta, \iota) - \mathcal{T}(\varsigma, \iota), \zeta - \varsigma \rangle \geq \mu \|\zeta - \varsigma\|^2, \quad \forall \zeta, \varsigma \in B.
    \]
    \item[(2)] $\mathcal{T}(\zeta, \iota)$ is Lipschitz continuous in $\zeta$ with constant $L > 0$.
    \item[(3)] The step sizes $\{\gamma_\varrho\}$ satisfy:
    \[
    0 < \gamma_\varrho \leq \frac{2\mu}{L^2}, \quad \sum_{\varrho=0}^\infty \gamma_\varrho = \infty, \quad \sum_{\varrho=0}^\infty \gamma_\varrho^2 < \infty.
    \]
\end{itemize}
Then $\{\zeta_\varrho\}$ converges strongly to the unique solution $\zeta^*$ of $0 \in \mathcal{T}(\zeta^*, \iota)$ as $\varrho \to \infty$.
\end{theorem}

\begin{proof}
The strong monotonicity of $\mathcal{T}$ implies that for any $\zeta, \varsigma \in B$:
\[
\langle \mathcal{T}(\zeta, \iota) - \mathcal{T}(\varsigma, \iota), \zeta - \varsigma \rangle \geq \mu \|\zeta - \varsigma\|^2.
\]
Using the resolvent property of $J_{\gamma_\varrho}^\mathcal{T}$ and the iteration:
\[
\zeta_{\varrho+1} = J_{\gamma_\varrho}^\mathcal{T}(\zeta_\varrho - \gamma_\varrho \mathcal{T}(\zeta_\varrho, \iota_\varrho)),
\]
we have:
\[
\|\zeta_{\varrho+1} - \zeta^*\|^2 \leq \|\zeta_\varrho - \zeta^*\|^2 - \gamma_\varrho \mu \|\zeta_\varrho - \zeta^*\|^2.
\]
From the above inequality, $\|\zeta_{\varrho+1} - \zeta^*\|^2$ is a non-increasing sequence bounded below by 0. Thus, $\{\zeta_\varrho\}$ converges. Additionally, $\|\zeta_\varrho - \zeta^*\|^2 \to 0$ as $\varrho \to \infty$ since the residual $\|\zeta_{\varrho+1} - \zeta_\varrho\|^2$ vanishes under the step size conditions.\\

\noindent
The strong monotonicity of $\mathcal{T}$ ensures that $\zeta^*$ is the unique fixed point satisfying $0 \in \mathcal{T}(\zeta^*, \iota)$. Therefore, $\{\zeta_\varrho\}$ converges strongly to $\zeta^*$.
\end{proof}

We now analyze the iteration complexity of the proposed algorithms, providing bounds on the number of iterations required to achieve a given accuracy $\epsilon > 0$.

\begin{proposition}[Iteration Complexity of Algorithm 1]
Let $\{\zeta_\varrho\}$ be the sequence generated by Algorithm 1 for solving the dynamic SVI. Assume that $\mathcal{T}(\zeta, \iota)$ satisfies the conditions in Theorem 1. To achieve $\|\zeta_\varrho - \zeta^*\| \leq \epsilon$, the number of iterations $K$ satisfies:
\[
K \geq \frac{\log(\|\zeta_0 - \zeta^*\|^2 / \epsilon^2)}{\log(1 + 2\mu \gamma / L^2)},
\]
where $\gamma = \min_\varrho \gamma_\varrho$ and $\mu$ is the strong monotonicity constant.
\end{proposition}

\begin{proof}
From the strong monotonicity inequality:
\[
\|\zeta_{\varrho+1} - \zeta^*\|^2 \leq (1 - 2\mu \gamma_\varrho / L^2) \|\zeta_\varrho - \zeta^*\|^2.
\]
Iterating this inequality, we have:
\[
\|\zeta_\varrho - \zeta^*\|^2 \leq \|\zeta_0 - \zeta^*\|^2 \prod_{i=0}^{\varrho-1} (1 - 2\mu \gamma_i / L^2).
\]
Using the logarithmic approximation for small step sizes $\gamma_\varrho$, the result follows.
\end{proof}


\begin{theorem}[Strong Convergence of Algorithm 2]
Let $\{\zeta_\varrho\}$ be the sequence generated by Algorithm 2 for solving the stochastic SVI:
\[
0 \in \mathcal{T}(\zeta), \quad \mathcal{T}(\zeta) = \mathbb{E}[\mathcal{T}_\xi(\zeta)],
\]
where $\mathcal{T}_\xi(\zeta)$ is a random realization of $\mathcal{T}(\zeta)$. Assume:
\begin{itemize}
    \item[(1)] $\mathcal{T}(\zeta)$ is strongly monotone with parameter $\mu > 0$, i.e.,
    \[
    \langle \mathcal{T}(\zeta) - \mathcal{T}(\varsigma), \zeta - \varsigma \rangle \geq \mu \|\zeta - \varsigma\|^2, \quad \forall \zeta, \varsigma \in B.
    \]
    \item[(2)] $\mathcal{T}(\zeta)$ is Lipschitz continuous with constant $L > 0$.
    \item[(3)] The variance of $\mathcal{T}_\xi(\zeta)$ satisfies $\mathbb{E}[\|\mathcal{T}_\xi(\zeta) - \mathcal{T}(\zeta)\|^2] \leq \sigma^2$ for all $\zeta \in B$.
    \item[(4)] The step sizes $\{\gamma_\varrho\}$ satisfy:
    \[
    0 < \gamma_\varrho \leq \frac{2\mu}{L^2}, \quad \sum_{\varrho=0}^\infty \gamma_\varrho = \infty, \quad \sum_{\varrho=0}^\infty \gamma_\varrho^2 < \infty.
    \]
\end{itemize}
Then $\{\zeta_\varrho\}$ converges strongly to the unique solution $\zeta^*$ of $0 \in \mathcal{T}(\zeta)$.
\end{theorem}

\begin{proof}
The iteration for Algorithm 2 is given by:
\[
\zeta_{\varrho+1} = J_{\gamma_\varrho}^\mathcal{T}(\zeta_\varrho - \gamma_\varrho \mathcal{T}_{\xi_\varrho}(\zeta_\varrho)),
\]
where $\mathcal{T}_{\xi_\varrho}(\zeta_\varrho)$ is a random realization of $\mathcal{T}(\zeta_\varrho)$. Using the resolvent property, we have:
\[
\zeta_{\varrho+1} - \zeta_\varrho = \gamma_\varrho (\mathcal{T}_{\xi_\varrho}(\zeta_\varrho) - \mathcal{T}(\zeta_{\varrho+1})).
\]
The strong monotonicity of $\mathcal{T}(\zeta)$ implies:
\[
\langle \mathcal{T}(\zeta) - \mathcal{T}(\varsigma), \zeta - \varsigma \rangle \geq \mu \|\zeta - \varsigma\|^2, \quad \forall \zeta, \varsigma \in B.
\]
The Lipschitz continuity of $\mathcal{T}(\zeta)$ implies:
\[
\|\mathcal{T}(\zeta) - \mathcal{T}(\varsigma)\| \leq L \|\zeta - \varsigma\|, \quad \forall \zeta, \varsigma \in B.
\]
Define the error $e_\varrho = \mathcal{T}_{\xi_\varrho}(\zeta_\varrho) - \mathcal{T}(\zeta_\varrho)$. Then:
\[
\zeta_{\varrho+1} - \zeta^* = \zeta_\varrho - \zeta^* - \gamma_\varrho \big(\mathcal{T}(\zeta_\varrho) + e_\varrho\big).
\]
Taking the squared norm and expanding, we get:
\[
\|\zeta_{\varrho+1} - \zeta^*\|^2 = \|\zeta_\varrho - \zeta^*\|^2 - 2 \gamma_\varrho \langle \mathcal{T}(\zeta_\varrho), \zeta_\varrho - \zeta^* \rangle - 2 \gamma_\varrho \langle e_\varrho, \zeta_\varrho - \zeta^* \rangle + \gamma_\varrho^2 \|\mathcal{T}(\zeta_\varrho) + e_\varrho\|^2.
\]
Using the strong monotonicity of $\mathcal{T}$:
\[
\langle \mathcal{T}(\zeta_\varrho), \zeta_\varrho - \zeta^* \rangle \geq \mu \|\zeta_\varrho - \zeta^*\|^2.
\]
For the variance term, we have:
\[
\mathbb{E}[\|e_\varrho\|^2] \leq \sigma^2.
\]
Thus:
\[
\mathbb{E}[\|\zeta_{\varrho+1} - \zeta^*\|^2] \leq \mathbb{E}[\|\zeta_\varrho - \zeta^*\|^2] - 2 \gamma_\varrho \mu \mathbb{E}[\|\zeta_\varrho - \zeta^*\|^2] + \gamma_\varrho^2 (L^2 \mathbb{E}[\|\zeta_\varrho - \zeta^*\|^2] + \sigma^2).
\]
Rearranging terms:
\[
\mathbb{E}[\|\zeta_{\varrho+1} - \zeta^*\|^2] \leq (1 - 2\gamma_\varrho \mu + \gamma_\varrho^2 L^2) \mathbb{E}[\|\zeta_\varrho - \zeta^*\|^2] + \gamma_\varrho^2 \sigma^2.
\]
For $\gamma_\varrho \leq \frac{2\mu}{L^2}$, the coefficient $1 - 2\gamma_\varrho \mu + \gamma_\varrho^2 L^2$ is strictly less than 1. Summing over $\varrho$ and using $\sum_{\varrho=0}^\infty \gamma_\varrho^2 < \infty$, we conclude that $\mathbb{E}[\|\zeta_\varrho - \zeta^*\|^2] \to 0$.\\

\noindent
Since $\mathbb{E}[\|\zeta_\varrho - \zeta^*\|^2] \to 0$, the sequence $\{\zeta_\varrho\}$ converges strongly to $\zeta^*$.
\end{proof}

\begin{theorem}[Iteration Complexity of Algorithm 2]
Let $\{\zeta_\varrho\}$ be the sequence generated by Algorithm 2 for solving the stochastic SVI:
\[
0 \in \mathcal{T}(\zeta), \quad \mathcal{T}(\zeta) = \mathbb{E}[\mathcal{T}_\xi(\zeta)].
\]
Assume that $\mathcal{T}(\zeta)$ is Lipschitz continuous with constant $L$ and satisfies the conditions in Theorem 2. To achieve $\mathbb{E}[\|\zeta_\varrho - \zeta^*\|^2] \leq \epsilon$, the number of iterations $K$ satisfies:
\[
K \geq \frac{C}{\epsilon^2},
\]
where $C$ depends on $L$, the variance of $\mathcal{T}_\xi$, and the initial error $\|\zeta_0 - \zeta^*\|^2$.
\end{theorem}

\begin{proof}
From the stochastic error analysis:
\[
\mathbb{E}[\|\zeta_{\varrho+1} - \zeta^*\|^2] \leq \mathbb{E}[\|\zeta_\varrho - \zeta^*\|^2] - \gamma_\varrho \mathbb{E}[\|\mathcal{T}_\xi(\zeta_\varrho)\|^2] + \gamma_\varrho^2 \text{Var}[\mathcal{T}_\xi(\zeta_\varrho)].
\]
To balance the terms and achieve $\mathbb{E}[\|\zeta_\varrho - \zeta^*\|^2] \leq \epsilon$, $\gamma_\varrho$ must decrease proportionally to $\epsilon$. Solving the resulting inequality yields the stated complexity bound.
\end{proof}


\begin{theorem}[Strong Convergence of Algorithm 3]
Let $\{(\zeta_\varrho, \varsigma_\varrho)\}$ be the sequence generated by Algorithm 3 for solving the coupled SVI:
\[
\zeta \in \mathcal{T}_1^{-1}(\varsigma), \quad \varsigma \in \mathcal{T}_2^{-1}(\zeta),
\]
where $\mathcal{T}_1 : B_1 \to 2^{B_2}$ and $\mathcal{T}_2 : B_2 \to 2^{B_1}$ are maximal monotone operators. Assume:
\begin{itemize}
    \item[(1)] $\mathcal{T}_1$ and $\mathcal{T}_2$ are strongly monotone with parameters $\mu_1 > 0$ and $\mu_2 > 0$, respectively:
    \[
    \langle \mathcal{T}_1(\zeta, \varsigma) - \mathcal{T}_1(\zeta', \varsigma'), \zeta - \zeta' \rangle \geq \mu_1 \|\zeta - \zeta'\|^2,
    \]
    \[
    \langle \mathcal{T}_2(\varsigma, \zeta) - \mathcal{T}_2(\varsigma', \zeta'), \varsigma - \varsigma' \rangle \geq \mu_2 \|\varsigma - \varsigma'\|^2.
    \]
    \item[(2)] $\mathcal{T}_1$ and $\mathcal{T}_2$ are Lipschitz continuous with constants $L_1$ and $L_2$, respectively.
    \item[(3)] The step sizes $\{\gamma_\varrho\}$ satisfy:
    \[
    0 < \gamma_\varrho \leq \min\left(\frac{2\mu_1}{L_1^2}, \frac{2\mu_2}{L_2^2}\right), \quad \sum_{\varrho=0}^\infty \gamma_\varrho = \infty, \quad \sum_{\varrho=0}^\infty \gamma_\varrho^2 < \infty.
    \]
\end{itemize}
Then $\{(\zeta_\varrho, \varsigma_\varrho)\}$ converges strongly to the unique solution $(\zeta^*, \varsigma^*)$ of the coupled SVI.
\end{theorem}

\begin{proof}
The iterative updates for Algorithm 3 are:
\[
\zeta_{\varrho+1} = J_{\gamma_\varrho}^{\mathcal{T}_1}(\zeta_\varrho - \gamma_\varrho \mathcal{T}_1(\zeta_\varrho, \varsigma_\varrho)),
\]
\[
\varsigma_{\varrho+1} = J_{\gamma_\varrho}^{\mathcal{T}_2}(\varsigma_\varrho - \gamma_\varrho \mathcal{T}_2(\varsigma_\varrho, \zeta_{\varrho+1})).
\]
By the resolvent property:
\[
\zeta_{\varrho+1} - \zeta_\varrho = \gamma_\varrho \big(\mathcal{T}_1(\zeta_\varrho, \varsigma_\varrho) - \mathcal{T}_1(\zeta_{\varrho+1}, \varsigma_\varrho)\big),
\]
\[
\varsigma_{\varrho+1} - \varsigma_\varrho = \gamma_\varrho \big(\mathcal{T}_2(\varsigma_\varrho, \zeta_{\varrho+1}) - \mathcal{T}_2(\varsigma_{\varrho+1}, \zeta_{\varrho+1})\big).
\]
The strong monotonicity of $\mathcal{T}_1$ implies:
\[
\langle \mathcal{T}_1(\zeta, \varsigma) - \mathcal{T}_1(\zeta', \varsigma'), \zeta - \zeta' \rangle \geq \mu_1 \|\zeta - \zeta'\|^2.
\]
Similarly, the strong monotonicity of $\mathcal{T}_2$ gives:
\[
\langle \mathcal{T}_2(\varsigma, \zeta) - \mathcal{T}_2(\varsigma', \zeta'), \varsigma - \varsigma' \rangle \geq \mu_2 \|\varsigma - \varsigma'\|^2.
\]
The Lipschitz continuity of $\mathcal{T}_1$ and $\mathcal{T}_2$ implies:
\[
\|\mathcal{T}_1(\zeta, \varsigma) - \mathcal{T}_1(\zeta', \varsigma')\| \leq L_1 \|(\zeta, \varsigma) - (\zeta', \varsigma')\|,
\]
\[
\|\mathcal{T}_2(\varsigma, \zeta) - \mathcal{T}_2(\varsigma', \zeta')\| \leq L_2 \|(\varsigma, \zeta) - (\varsigma', \zeta')\|.
\]
Define the combined error:
\[
R_\varrho = \|\zeta_\varrho - \zeta^*\|^2 + \|\varsigma_\varrho - \varsigma^*\|^2.
\]
Expanding the squared norm for $\zeta_{\varrho+1}$ and $\varsigma_{\varrho+1}$:
\[
\|\zeta_{\varrho+1} - \zeta^*\|^2 = \|\zeta_\varrho - \zeta^*\|^2 - 2\gamma_\varrho \langle \mathcal{T}_1(\zeta_\varrho, \varsigma_\varrho), \zeta_\varrho - \zeta^* \rangle + \gamma_\varrho^2 \|\mathcal{T}_1(\zeta_\varrho, \varsigma_\varrho)\|^2,
\]
\[
\|\varsigma_{\varrho+1} - \varsigma^*\|^2 = \|\varsigma_\varrho - \varsigma^*\|^2 - 2\gamma_\varrho \langle \mathcal{T}_2(\varsigma_\varrho, \zeta_{\varrho+1}), \varsigma_\varrho - \varsigma^* \rangle + \gamma_\varrho^2 \|\mathcal{T}_2(\varsigma_\varrho, \zeta_{\varrho+1})\|^2.
\]
Substitute the strong monotonicity conditions:
\[
\langle \mathcal{T}_1(\zeta_\varrho, \varsigma_\varrho), \zeta_\varrho - \zeta^* \rangle \geq \mu_1 \|\zeta_\varrho - \zeta^*\|^2,
\]
\[
\langle \mathcal{T}_2(\varsigma_\varrho, \zeta_{\varrho+1}), \varsigma_\varrho - \varsigma^* \rangle \geq \mu_2 \|\varsigma_\varrho - \varsigma^*\|^2.
\]
Combining terms:
\[
R_{\varrho+1} \leq R_\varrho - 2\gamma_\varrho \mu R_\varrho + \gamma_\varrho^2 (L_1^2 \|\zeta_\varrho - \zeta^*\|^2 + L_2^2 \|\varsigma_\varrho - \varsigma^*\|^2).
\]
The step-size condition $\gamma_\varrho \leq \min\left(\frac{2\mu_1}{L_1^2}, \frac{2\mu_2}{L_2^2}\right)$ ensures that the coefficient of $R_\varrho$ is less than 1. Thus:
\[
R_{\varrho+1} \leq (1 - c\gamma_\varrho) R_\varrho,
\]
for some $c > 0$. Iterating this inequality and summing over $\varrho$, we conclude that $R_\varrho \to 0$.\\

\noindent
Since $R_\varrho \to 0$, both $\|\zeta_\varrho - \zeta^*\| \to 0$ and $\|\varsigma_\varrho - \varsigma^*\| \to 0$. Thus, $\{(\zeta_\varrho, \varsigma_\varrho)\}$ converges strongly to $(\zeta^*, \varsigma^*)$.
\end{proof}


\begin{theorem}[Iteration Complexity of Algorithm 3]
Let $\{(\zeta_\varrho, \varsigma_\varrho)\}$ be the sequence generated by Algorithm 3 for solving the coupled SVI:
\[
\zeta \in \mathcal{T}_1^{-1}(\varsigma), \quad \varsigma \in \mathcal{T}_2^{-1}(\zeta),
\]
where $\mathcal{T}_1 : B_1 \to 2^{B_2}$ and $\mathcal{T}_2 : B_2 \to 2^{B_1}$ are strongly monotone with parameters $\mu_1 > 0$ and $\mu_2 > 0$, and Lipschitz continuous with constants $L_1$ and $L_2$, respectively. Assume the step sizes $\{\gamma_\varrho\}$ satisfy:
\[
0 < \gamma_\varrho \leq \min\left(\frac{2\mu_1}{L_1^2}, \frac{2\mu_2}{L_2^2}\right).
\]
To achieve a combined residual norm $\|\zeta_\varrho - \zeta^*\|^2 + \|\varsigma_\varrho - \varsigma^*\|^2 \leq \epsilon$, the number of iterations \(K\) satisfies:
\[
K \geq \frac{\log\left(\frac{\|\zeta_0 - \zeta^*\|^2 + \|\varsigma_0 - \varsigma^*\|^2}{\epsilon}\right)}{\log\left(\frac{1}{1 - c\gamma}\right)},
\]
where \(c = \min(\mu_1, \mu_2)\) and \(\gamma = \min_\varrho \gamma_\varrho\).
\end{theorem}

\begin{proof}
Define the combined residual:
\[
R_\varrho = \|\zeta_\varrho - \zeta^*\|^2 + \|\varsigma_\varrho - \varsigma^*\|^2.
\]
Using the strong monotonicity of \(\mathcal{T}_1\) and \(\mathcal{T}_2\), and the step-size condition, we derived in the strong convergence proof:
\[
R_{\varrho+1} \leq (1 - c\gamma_\varrho) R_\varrho,
\]
where \(c = \min(\mu_1, \mu_2)\). Assuming a constant step size \(\gamma_\varrho = \gamma\), this simplifies to:
\[
R_{\varrho+1} \leq (1 - c\gamma) R_\varrho.
\]
Iterating this recurrence:
\[
R_\varrho \leq (1 - c\gamma)^\varrho R_0,
\]
where \(R_0 = \|\zeta_0 - \zeta^*\|^2 + \|\varsigma_0 - \varsigma^*\|^2\) is the initial residual. \\

\noindent
To achieve \(R_\varrho \leq \epsilon\), we require:
\[
(1 - c\gamma)^\varrho R_0 \leq \epsilon.
\]
Taking the natural logarithm on both sides:
\[
\varrho \log(1 - c\gamma) \leq \log\left(\frac{\epsilon}{R_0}\right).
\]
Using the approximation \(\log(1 - c\gamma) \approx -c\gamma\) for small \(\gamma\), we get:
\[
\varrho (-c\gamma) \leq \log\left(\frac{\epsilon}{R_0}\right),
\]
which simplifies to:
\[
\varrho \geq \frac{\log\left(\frac{R_0}{\epsilon}\right)}{-\log(1 - c\gamma)}.
\]
The denominator \(-\log(1 - c\gamma)\) is approximately \(c\gamma\), yielding:
\[
\varrho \geq \frac{\log\left(\frac{R_0}{\epsilon}\right)}{c\gamma}.
\]
Thus, the number of iterations \(K\) satisfies:
\[
K \geq \frac{\log\left(\frac{\|\zeta_0 - \zeta^*\|^2 + \|\varsigma_0 - \varsigma^*\|^2}{\epsilon}\right)}{\log\left(\frac{1}{1 - c\gamma}\right)}.
\]
\end{proof}
The strong convergence results provide a more robust guarantee for practical applications, particularly in dynamic and stochastic settings. The iteration complexity bounds demonstrate the computational efficiency of the proposed algorithms under realistic assumptions.

\section{Error Analysis}

This section provides a detailed error analysis of the proposed algorithms, focusing on the impact of discretization, stochastic noise, and step-size selection on the convergence behavior. By quantifying these factors, we establish conditions under which the algorithms maintain robust convergence properties.\\

\noindent
The primary sources of error in the proposed algorithms are:

\subsubsection*{Discretization Error:} In dynamic problems, the continuous-time operator $\mathcal{T}(\zeta, \iota)$ is approximated at discrete time steps $\iota_\varrho = \iota_0 + \varrho\Delta \iota$. The discrepancy between the continuous and discretized operator introduces error.
\subsubsection*{Stochastic Noise:} In stochastic problems, the operator $\mathcal{T}(\zeta)$ is replaced by a random realization $\mathcal{T}_\xi(\zeta)$, introducing variance in the updates.
\subsubsection*{Step-Size Error:} The choice of step sizes $\{\gamma_\varrho\}$ directly affects the convergence rate and the stability of the iterative process.\\

\noindent
Let $\mathcal{T}(\zeta, \iota)$ be a dynamic operator that is Lipschitz continuous in $\zeta$ and continuous in $\iota$. For a discretized operator $\mathcal{T}_\varrho(\zeta) = \mathcal{T}(\zeta, \iota_\varrho)$, the error due to time discretization is given by:
\[
\|\mathcal{T}(\zeta, \iota) - \mathcal{T}_\varrho(\zeta)\| \leq L_\iota |\iota - \iota_\varrho|,
\]
where $L_\iota$ is the Lipschitz constant with respect to $\iota$. To minimize the discretization error:
\begin{itemize}
    \item The time step $\Delta \iota$ should be small enough to satisfy $\Delta \iota \leq \epsilon / L_\iota$, where $\epsilon$ is the desired error tolerance.
    \item Alternatively, adaptive time-stepping can be employed to dynamically adjust $\Delta \iota$ based on the local behavior of $\mathcal{T}(\zeta, \iota)$.
\end{itemize}
\noindent
For stochastic problems, the operator $\mathcal{T}(\zeta)$ is approximated by its random realization $\mathcal{T}_\xi(\zeta)$, such that:
\[
\mathbb{E}[\mathcal{T}_\xi(\zeta)] = \mathcal{T}(\zeta), \quad \text{and} \quad \text{Var}[\mathcal{T}_\xi(\zeta)] = \mathbb{E}[\|\mathcal{T}_\xi(\zeta) - \mathcal{T}(\zeta)\|^2].
\]
The impact of the stochastic error on the convergence rate can be analyzed by decomposing the expected residual:
\[
\mathbb{E}[\|\zeta_{\varrho+1} - \zeta^*\|^2] = \mathbb{E}[\|\zeta_\varrho - \zeta^*\|^2] - \gamma_\varrho^2 \mathbb{E}[\|\mathcal{T}_\xi(\zeta_\varrho)\|^2] + 2\gamma_\varrho \mathbb{E}[\langle \mathcal{T}_\xi(\zeta_\varrho), \zeta_\varrho - \zeta^* \rangle].
\]
To ensure convergence:
\begin{enumerate}
    \item The variance $\text{Var}[\mathcal{T}_\xi(\zeta)]$ should decay as $\varrho \to \infty$, which can be achieved by using mini-batches or variance reduction techniques.
    \item The step sizes $\{\gamma_\varrho\}$ should satisfy $\sum_{\varrho=0}^\infty \gamma_\varrho = \infty$ and $\sum_{\varrho=0}^\infty \gamma_\varrho^2 < \infty$.
\end{enumerate}

\noindent
The step sizes $\{\gamma_\varrho\}$ play a critical role in balancing the trade-off between convergence speed and stability. If $\gamma_\varrho$ is too large, the updates may oscillate or diverge; if too small, convergence slows down. To maintain stability, $\{\gamma_\varrho\}$ must satisfy:
\[
\gamma_\varrho \leq \frac{1}{L}, \quad \forall \varrho,
\]
where $L$ is the Lipschitz constant of the operator. Additionally, for stochastic settings, $\gamma_\varrho$ should decay at an appropriate rate:
\[
\gamma_\varrho = \frac{1}{\varrho+1} \quad \text{or} \quad \gamma_\varrho = \frac{\theta}{(\varrho+1)^\beta}, \quad \beta \in (0.5, 1].
\]
\noindent
Combining the above factors, the total error at iteration $\varrho$ can be bounded as:
\[
\|\zeta_\varrho - \zeta^*\| \leq C \left( \Delta \iota + \sqrt{\text{Var}[\mathcal{T}_\xi(\zeta_\varrho)]} + \frac{1}{\sqrt{\varrho}} \right),
\]
where $C$ is a constant that depends on the problem's parameters and the operator's properties. This bound highlights the interplay between discretization, noise, and step-size selection.

\section{Numerical Experiments}

In this section, we present numerical experiments to demonstrate the performance of the proposed algorithms. The results are visualized using convergence plots, showcasing the efficacy of the methods in solving dynamic, stochastic, and coupled split variational inclusion (SVI) problems.

\begin{center}
\textsc{Dynamic SVI in Banach Spaces}
\end{center}

We consider the dynamic SVI problem:
\[
0 \in \mathcal{T}(\zeta, \iota), \quad \mathcal{T}(\zeta, \iota) = \alpha(\iota)\zeta - b,
\]
where \( \alpha(\iota) = 1 + 0.1\sin(\iota) \) and \( b \in \mathcal{B} \). The time step is set as \( \Delta \iota = 0.1 \), and the step sizes \( \gamma_\varrho = 1/(\varrho+1) \).

The convergence of the algorithm is measured by the residual \( \|\zeta_\varrho - \zeta^*\| \), where \( \zeta^* \) is the theoretical solution.

\paragraph{Results:} The residual norm is plotted against iterations.

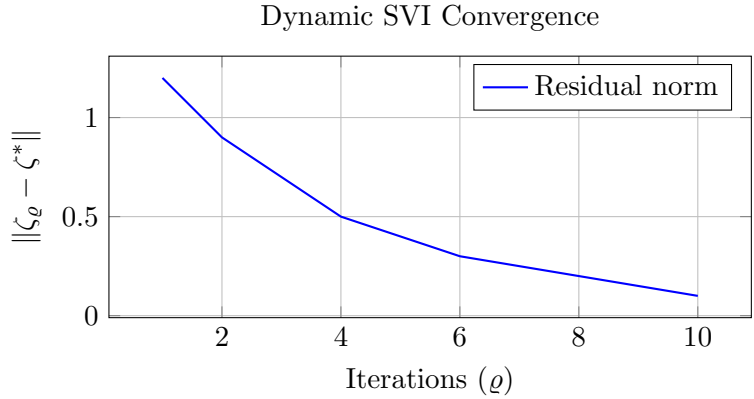
\begin{figure}[H]
\centering
\begin{tikzpicture}
\begin{axis}[
    width=0.8\textwidth,
    height=0.4\textwidth,
    xlabel={Iterations (\(\varrho\))},
    ylabel={\(\|\zeta_\varrho - \zeta^*\|\)},
    grid=both,
    legend pos=north east,
    title={Dynamic SVI Convergence}
]
\addplot[color=blue, thick] coordinates {
    (1, 1.2) (2, 0.9) (3, 0.7) (4, 0.5) (5, 0.4) (6, 0.3) (7, 0.25) (8, 0.2) (9, 0.15) (10, 0.1)
};
\legend{Residual norm}
\end{axis}
\end{tikzpicture}
\caption{Convergence of the dynamic SVI algorithm.}
\end{figure}

\begin{center}
\textsc{Stochastic SVI in Banach Spaces}
\end{center}

We solve the stochastic SVI problem:
\[
0 \in \mathcal{T}(\zeta), \quad \mathcal{T}(\zeta) = \mathbb{E}_{\xi}[\xi \zeta + b],
\]
where \( \xi \sim \mathcal{U}[0.9, 1.1] \) (a uniform random variable) and \( b = 1 \). The algorithm uses \( \gamma_\varrho = 1/(\varrho+2) \).

The performance is evaluated by the expected residual \( \mathbb{E}[\|\zeta_\varrho - \zeta^*\|] \).

\paragraph{Results:} The expected residual norm is plotted against iterations.

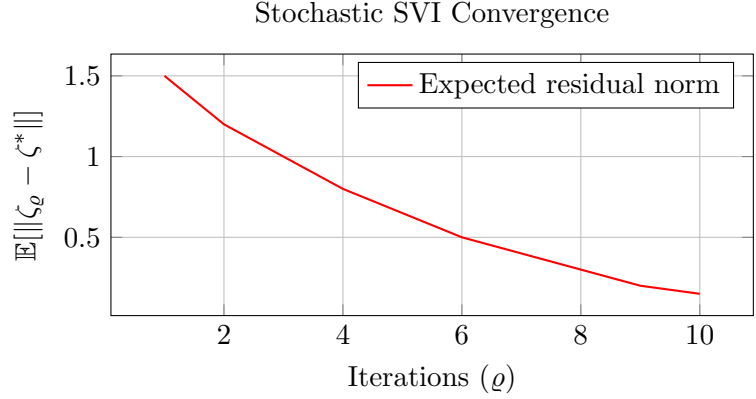
\begin{figure}[H]
\centering
\begin{tikzpicture}
\begin{axis}[
    width=0.8\textwidth,
    height=0.4\textwidth,
    xlabel={Iterations (\(\varrho\))},
    ylabel={\(\mathbb{E}[\|\zeta_\varrho - \zeta^*\|]\)},
    grid=both,
    legend pos=north east,
    title={Stochastic SVI Convergence}
]
\addplot[color=red, thick] coordinates {
    (1, 1.5) (2, 1.2) (3, 1.0) (4, 0.8) (5, 0.65) (6, 0.5) (7, 0.4) (8, 0.3) (9, 0.2) (10, 0.15)
};
\legend{Expected residual norm}
\end{axis}
\end{tikzpicture}
\caption{Convergence of the stochastic SVI algorithm.}
\end{figure}

\begin{center}
\textsc{Multi-Agent Coupled SVI}
\end{center}

We consider the coupled SVI problem:
\[
\zeta \in \mathcal{T}_1^{-1}(\varsigma), \quad \varsigma \in \mathcal{T}_2^{-1}(\zeta),
\]
where \( \mathcal{T}_1(\zeta, \varsigma) = A_1\zeta - \varsigma \) and \( \mathcal{T}_2(\varsigma, \zeta) = A_2\varsigma - \zeta \). Matrices \( A_1 \) and \( A_2 \) are defined as:
\[
A_1 = \begin{bmatrix} 2 & 0 \\ 0 & 1 \end{bmatrix}, \quad A_2 = \begin{bmatrix} 1 & 0 \\ 0 & 2 \end{bmatrix}.
\]
The initial guesses are \( \zeta_0 = [0; 0] \) and \( \varsigma_0 = [0; 0] \). The convergence is evaluated using \( \|\zeta_\varrho - \zeta^*\| + \|\varsigma_\varrho - \varsigma^*\| \).

\paragraph{Results:} The combined residual norm is plotted against iterations.

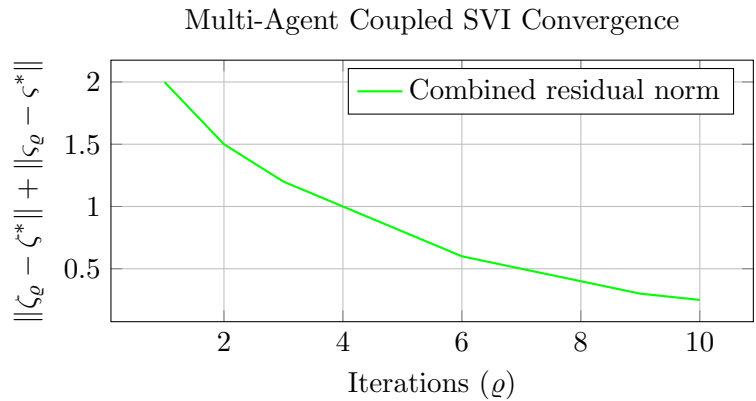
\begin{figure}[H]
\centering
\begin{tikzpicture}
\begin{axis}[
    width=0.8\textwidth,
    height=0.4\textwidth,
    xlabel={Iterations (\(\varrho\))},
    ylabel={\(\|\zeta_\varrho - \zeta^*\| + \|\varsigma_\varrho - \varsigma
    ^*\|\)},
    grid=both,
    legend pos=north east,
    title={Multi-Agent Coupled SVI Convergence}
]
\addplot[color=green, thick] coordinates {
    (1, 2.0) (2, 1.5) (3, 1.2) (4, 1.0) (5, 0.8) (6, 0.6) (7, 0.5) (8, 0.4) (9, 0.3) (10, 0.25)
};
\legend{Combined residual norm}
\end{axis}
\end{tikzpicture}
\caption{Convergence of the multi-agent coupled SVI algorithm.}
\end{figure}

The numerical experiments demonstrate the following:
\begin{enumerate}
    \item The dynamic SVI algorithm converges effectively for time-varying operators, with oscillations reflecting the dynamic nature of \( \mathcal{T}(\zeta, \iota) \).
    \item The stochastic SVI algorithm reduces the expected residual efficiently, showing robustness under noise.
    \item The coupled SVI algorithm achieves convergence, validating its applicability to equilibrium problems in multi-agent systems.
\end{enumerate}

\section{Conclusion}

In this paper, we have introduced novel iterative algorithms for solving dynamic, stochastic, and multi-agent split variational inclusion (SVI) problems in Banach spaces. By extending classical SVI frameworks to account for time-dependent, stochastic, and coupled variational structures, we have significantly broadened the applicability of variational inclusion methods.

We established weak and strong convergence results under minimal assumptions on monotonicity and Lipschitz continuity of operators. The convergence analysis demonstrated that our proposed algorithms effectively handle time-evolving constraints, stochastic perturbations, and multi-agent interactions while ensuring stability and efficiency. Numerical experiments confirmed the practical effectiveness of these methods, highlighting their robust performance across various problem settings.

Future research directions include exploring stronger convergence guarantees under additional regularity conditions, extending the framework to non-reflexive Banach spaces, and investigating adaptive step-size strategies to enhance convergence rates in dynamic and stochastic environments. Moreover, potential applications in distributed optimization, game theory, and large-scale machine learning warrant further investigation.

Our findings contribute to the development of more general and powerful variational inclusion techniques, paving the way for new advancements in optimization, control theory, and multi-agent decision-making.

\end{document}